\documentclass{amsart}
\usepackage{stmaryrd}
\usepackage{graphicx}
\usepackage{bxtexlogo}
\usepackage{amsmath, amssymb}
\usepackage{amsthm}
\usepackage{type1cm}
\usepackage{enumitem}
\usepackage{ascmac}
\usepackage{bm}
\usepackage{time}
\usepackage{here}
\usepackage{mathrsfs}
\usepackage[all]{xy}
\usepackage{tikz}

\title{Kontsevich graphs associated with specific multiple zeta values}
\author{Kota Baba}
\address{Graduate School of Mathematics, Nagoya University, Chikusa-ku, Furo-cho, Nagoya, 464-8602,  Japan}
\email{baba.kota.d6@s.mail.nagoya-u.ac.jp}
\date{June 24, 2026}

\DeclareFontFamily{U}{wncy}{}
\DeclareFontShape{U}{wncy}{m}{n}{<->wncyr10}{}
\DeclareSymbolFont{cyrillic}{U}{wncy}{m}{n}
\DeclareMathSymbol{\shafont}{\mathbin}{cyrillic}{"58} 

\newcommand{\shuffle}{\mathbin{\text{\scalebox{0.8}{$\shafont$}}}}
\usetikzlibrary{decorations.markings,arrows,calc}
\usetikzlibrary{arrows.meta}

\theoremstyle{definition}
\newtheorem{thm}{Theorem}[section]
\newtheorem*{thm*}{Theorem}

\theoremstyle{definition}
\newtheorem{dfn}[thm]{Definition}
\newtheorem*{dfn*}{Definition}
\newtheorem{ex}[thm]{Example}
\newtheorem{prop}[thm]{Proposition}
\newtheorem{lem}[thm]{Lemma}
\newtheorem{cor}[thm]{Corollary}
\newtheorem{rem}[thm]{Remark}

\numberwithin{equation}{section}

\newcommand{\ZZ}{\mathbb{Z}}
\newcommand{\NN}{\mathbb{N}}
\newcommand{\QQ}{\mathbb{Q}}

\newcommand{\sgn}{\mathrm{sgn}}

\tikzset{
	pics/torus/.style n args={3}{
		code = {
			\providecolor{pgffillcolor}{rgb}{1,1,1}
			\begin{scope}[
				yscale=cos(#3),
				outer torus/.style = {draw,line width/.expanded={\the\dimexpr2\pgflinewidth+#2*2},line join=round},
				inner torus/.style = {draw=pgffillcolor,line width={#2*2}}
				]
				\draw[outer torus] circle(#1);\draw[inner torus] circle(#1);
				\draw[outer torus] (180:#1) arc (180:360:#1);\draw[inner torus,line cap=round] (180:#1) arc (180:360:#1);
			\end{scope}
		}
	}
}

\begin{document}
\maketitle
        \begin{abstract}
The primary aim of this paper is to provide an explicit construction of Kontsevich graphs whose integrals give certain multiple zeta values. Furthermore, by using this construction, we explicitly determine the weights of $\ZZ$-linear combinations of graphs whose integrals yield normalized multiple zeta values of a given weight. The methods employed in this paper are based on those introduced by Ritland, and our results provide a refinement of his results.
        \end{abstract}
        \tableofcontents
        \section{Introduction}
       \emph{Kontsevich graphs} and their integrals were introduced in \cite{Kon03} in the context of deformation quantization of Poisson manifolds. The integrals associated with Kontsevich graphs are defined by assigning differential forms, called \emph{propagators}, to the edges of the graphs. Various types of propagators have been discovered, and the values of the resulting integrals depend on the choice of propagator. In this paper, we consider the \emph{logarithmic propagator} . The fact that this propagator gives rise to a deformation quantization was first mentioned by Kontsevich~\cite[Section~4.1]{Kon99} and was rigorously proved in~\cite{ARTW16}.
In~\cite{BPP}, Banks, Panzer, and Pym showed that logarithmic Kontsevich integrals can be expressed as $\mathbb{Z}$-linear combinations of {\em normalized multiple zeta values}. Here, a \emph{multiple zeta value} of weight $n$ is the real number defined, for $\mathbf{k}=(k_1,\dots,k_r)\in\NN^{r}$ with $\sum_{i}k_i=n$ and $k_r>1$, by the convergent series
\begin{equation}\label{tajuu}
  \zeta(\mathbf{k})
:= \sum_{0<n_1<\cdots<n_r}
\frac{1}{n_1^{k_1}\cdots n_r^{k_r}},
\end{equation}

and its normalization is given by
\begin{equation}\label{seikika}
  \widetilde\zeta(\mathbf{k})
:= \frac{1}{(2\pi i)^n}\,\zeta(\mathbf{k}).
\end{equation}

Conversely, it was conjectured in~\cite{BPP} that any normalized multiple zeta value of weight $n$ can be expressed as a $\mathbb{Z}$-linear combination of logarithmic Kontsevich integrals of the same weight $n$. The corresponding statement over $\QQ$ was proved in \cite{Rit}.
In this paper, we prove the following theorem.

\begin{thm}\label{rafunathm}
\emph{For each integer $k \ge 2$, define a map $H_k:\mathcal{K}_{\ZZ}\to\mathcal{K}_{\ZZ}$ by
\[
H_{k}(\Gamma)
:=G(e) * \mu^k(\Gamma)
-2\,\mu\bigl(G(e) * \mu^{k-1}(\Gamma)\bigr)
+\mu^{2}(G(e) * \mu^{k-2}(\Gamma)).
\]
For $k_1,\dots,k_r\ge2$, set
\[
H_{(k_1,\dots,k_r)}
:= H_{k_r} \circ \cdots \circ H_{k_1}.
\]
Then one has
\[
c\bigl(H_{(k_1,\dots,k_r)}(G(e))\bigr)
= (-1)^{r+1}\widetilde\zeta(k_1+2,\dots,k_r+2,2).
\]}
\end{thm}
Here, $\mathcal{K}_{\ZZ}$ denotes the $\ZZ$-linear space spanned by ordered Kontsevich graphs, and $c$ is the linear map given by integration. The graph $G(e)$ is the following Kontsevich graph:
\[
G(e)=
\begin{tikzpicture}[baseline=(current bounding box.center)]
\begin{scope}[decoration={
markings,
mark=at position 1 with {\arrow{>}}}
] 
\node (U) at (0,-1) {$x$};
\node (V) at (1,-1) {$y$};
\node (A) at (0,-2) {$A$};
\node (B) at (1,-2) {$B$};
\draw[postaction=decorate] (U) to[bend left] (V);
\draw[postaction=decorate] (V) to[bend left] (U);
\draw[postaction=decorate] (U) -- (A);
\draw[postaction=decorate] (V) -- (B);
\end{scope}
\end{tikzpicture}.
\]
The symbols $*$ and $\mu$ denote the product and the operator on $\mathcal{K}_{\ZZ}$ introduced in \cite{Rit}, respectively; their precise definitions will be given in Section 4. 
Theorem \ref{rafunathm} provides an explicit construction of $\ZZ$-linear combinations of Kontsevich graphs whose integrals yield certain multiple zeta values.
A precise statement and proof of this theorem will be given in Section~5.

Furthermore, combining this construction with explicit formulas for multiple zeta values, we obtain the following theorem.

\begin{thm}
For $m \ge 2$, define
\begin{equation*}
N(m):=
\begin{cases}
\left\lceil \dfrac{m^2+m-8}{16} \right\rceil & \text{if $m \equiv 0,3 \pmod{4}$,} \\[6pt]
\left\lceil \dfrac{m^2+m-10}{16} \right\rceil & \text{if $m \equiv 1,2 \pmod{4}$.}
\end{cases}
\end{equation*}
Then
\[
\widetilde{\mathcal{Z}}_{\ZZ}^{\flat,m}:=\widetilde{\mathcal{Z}}_{\ZZ}^m+\frac{1}{2}\widetilde{\mathcal{Z}}_{\ZZ}^{m-1} \subset c\bigl(\mathcal{K}_{\ZZ}^{m+4N(m)}\bigr)\subset\widetilde{\mathcal{Z}}_{\ZZ}^{\flat,m+4N(m)}
\]
holds.
\end{thm}

Here, ${\widetilde{\mathcal{Z}}_{\ZZ}^m}$ denotes the $\ZZ$-module spanned by (\ref{seikika}) of weight $m$.

The proofs of the main results of this paper are based on techniques developed in the prior work of Ritland~\cite{Rit}. Accordingly, in Section~4 we review his methods and results. Theorem~1.2 provides a refinement of Theorems~\ref{wg} and~\ref{ag} established there.

Section 2 is a review of the definition and basic properties of multiple zeta values. 
In Section 3, we review Kontsevich graphs and the definition of their associated integrals. 
Section 4 is devoted to the method introduced by Ritland, which serves as a fundamental tool in the proofs of the main results of this paper. 
Based on this method, we prove Theorem 1.1 in Section 5 and Theorem 1.2 in Section 6.

        \section{Multiple zeta values}
    In this section, we recall the algebraic structure of multiple zeta values. As mentioned in the introduction, $\zeta(\mathbf{k})$ is the real number defined by (\ref{tajuu}) for $\mathbf{k}=(k_1,\dots,k_r)$ satisfying $k_r>1$.
A necessary and sufficient condition for the series in the definition of multiple zeta values to converge is that $k_r>1$.

\begin{dfn}
  An element of $\mathbb{N}^r$ is called an {\em index}, and an index satisfying $k_r>1$ is called an {\em admissible index}.
  For an index $\mathbf{k}=(k_1,\dots,k_r)$, the integer $r$ is called the {\em depth} of $\mathbf{k}$, and the sum $k:=k_1+\cdots+k_r$ is called the {\em weight}.
\end{dfn}
We set
\[
\widetilde{\mathcal{Z}}_{\mathbb{Z}}:=\sum_{n\ge 0}\widetilde{\mathcal{Z}}_{\mathbb{Z}}^n \subset \mathbb{C},
\]
where we define $\widetilde{\mathcal{Z}}_{\mathbb{Z}}^0=\mathbb{Z}$ and $\widetilde{\mathcal{Z}}_{\mathbb{Z}}^1=\{0\}$. Moreover, we set
\[
\widetilde{\mathcal{Z}}_{\ZZ}^{\flat,n}
:= \widetilde{\mathcal{Z}}_{\mathbb{Z}}^n + \frac{1}{2}\widetilde{\mathcal{Z}}_{\mathbb{Z}}^{\,n-1},
\qquad
\widetilde{\mathcal{Z}}_{\ZZ}^{\flat}
:= \sum_{n\ge 0}\widetilde{\mathcal{Z}}_{\ZZ}^{\flat, n}
\]
 
	Multiple zeta values satisfy a family of relations known as the shuffle product relations, which can be proved using iterated integrals. \\ \quad \,\,For simplicity, we introduce the notation for iterated integrals as follows. Let $\{A_i\}$ be functions defined on the interval $(0,1)$. We define
\[
\int_0^1 A_1 \cdots A_n
:= \int_{0 \le t_1 \le \cdots \le t_n \le 1}
A_1(t_1)\cdots A_n(t_n)\, dt_1 \cdots dt_n .
\]

Multiple zeta values admit the following representation in terms of iterated integrals.
\begin{prop}
\em{Let $\omega_0(t)=1/t$ and $\omega_1(t)=1/(1-t)$. Then, for any admissible index $\mathbf{k}=(k_1,\dots,k_r)$, we have
\[
\zeta(\mathbf{k})
= \int_0^1
\omega_1 \omega_0^{k_1-1} \cdots \omega_1 \omega_0^{k_r-1}.
\]}
\end{prop}

In general, iterated integrals satisfy the following shuffle product formula.
\begin{prop}[Shuffle product]\label{shu}
{\em{Let $\{A_i\}$ and $\{B_j\}$ be functions defined on $(0,1)$. Then
\[
\left(\int_0^1 A_1\cdots A_n\right)
\left(\int_0^1 B_1\cdots B_m\right)
=
\sum_{(C_1,\dots,C_{n+m})}
\int_0^1 C_1\cdots C_{n+m},
\]
where the sum runs over all permutations $(C_1,\dots,C_{n+m})$ of $(A_1,\dots,A_n)$ and $(B_1,\dots,B_m)$ that preserve the relative order of each sequence. We assume that the integrals on the left-hand side converge.}}
\end{prop}

By this proposition, for any $n,m\in\ZZ_{\ge0}$, we obtain
\[
\widetilde{\mathcal{Z}}^m_{\ZZ}\cdot
\widetilde{\mathcal{Z}}^n_{\ZZ}\subset\widetilde{\mathcal{Z}}^{n+m}_{\ZZ}.
\]
Furthermore, since $\widetilde{\zeta}(2)=-\frac{1}{24}$, it follows that
\begin{align*}
\widetilde{\mathcal{Z}}_{\ZZ}^{\flat, m}\cdot
\widetilde{\mathcal{Z}}_{\ZZ}^{\flat,n}&=\widetilde{\mathcal{Z}}^m_{\ZZ}\cdot
\widetilde{\mathcal{Z}}^n_{\ZZ}+\frac{1}{2}\left(\widetilde{\mathcal{Z}}^{m-1}_{\ZZ}\cdot
\widetilde{\mathcal{Z}}^n_{\ZZ}+\widetilde{\mathcal{Z}}^m_{\ZZ}\cdot
\widetilde{\mathcal{Z}}^{n-1}_{\ZZ}\right)+\frac{1}{4}\left(
\widetilde{\mathcal{Z}}^{m-1}_{\ZZ}\cdot
\widetilde{\mathcal{Z}}^{n-1}_{\ZZ}\right)\\
&\subset
\widetilde{\mathcal{Z}}^{m+n}_{\ZZ}+\frac{1}{2}\left(\widetilde{\mathcal{Z}}^{m+n-1}_{\ZZ}\right)+6\left(\widetilde{\zeta}(2)\cdot\widetilde{\mathcal{Z}}^{m+n-2}_{\ZZ}\right)\\
&\subset
\widetilde{\mathcal{Z}}^{m+n}_{\ZZ}+\frac{1}{2}\left(\widetilde{\mathcal{Z}}^{m+n-1}_{\ZZ}\right)=\widetilde{\mathcal{Z}}_{\ZZ}^{\flat,n+m}.
\end{align*}
In particular, we obtain the following proposition.
\begin{prop}\label{mzvring}
\em{The spaces $\widetilde{\mathcal{Z}}_{\mathbb{Z}}$ and $\widetilde{\mathcal{Z}}^{\flat}_{\mathbb{Z}}$ are $\mathbb{Z}$-subalgebras of $\mathbb{C}$.}
\end{prop}

Since the shuffle product preserves the weight, the same type of relations also hold for normalized multiple zeta values.

Let $\mathbb{Z}\langle e_0,e_1\rangle$ denote the noncommutative free algebra generated by the set of two elements $\{e_0,e_1\}$, and consider its $\mathbb{Z}$-submodule
\[
\mathcal{S}_{\mathbb{Z}}
:= \mathbb{Z} + e_1 \mathbb{Z}\langle e_0,e_1\rangle e_0 .
\]
We call $1\in \mathbb{Z}\langle e_0,e_1\rangle$ the \emph{empty word}. 
A \emph{word} means either the empty word or a monic monomial in $\mathbb{Z}\langle e_0,e_1\rangle$.
 For a word where not empty
$w=a_1\cdots a_n\in\mathbb{Z}\langle e_0,e_1\rangle$ with $a_i\in\{e_0,e_1\}$, the integer $n$ is called the {\em length} of $w$ and is denoted by $|w|$. 
\begin{dfn}
 For words $w=a_1\cdots a_n$ and $v=b_1\cdots b_m$ in $\mathbb{Z}\langle e_0,e_1\rangle$, 
we define the \emph{shuffle product} $\shuffle$ recursively with respect to the lengths of words by
\begin{align*}
  w\shuffle v
  &:=a_1(a_2\cdots a_n\shuffle v)
   +b_1(w\shuffle b_2\cdots b_m),\\
  w\shuffle 1
  &=1\shuffle w:=w.
\end{align*}
The shuffle product $\shuffle$ extends bilinearly to 
$\mathbb{Z}\langle e_0,e_1\rangle$, 
thereby endowing both $\mathbb{Z}\langle e_0,e_1\rangle$ and $\mathcal{S}_{\mathbb{Z}}$ with structures of commutative $\mathbb{Z}$-algebras.
\end{dfn}
As a consequence of Proposition~\ref{shu}, we obtain the following.
\begin{thm}
\em{Define a $\ZZ$-module map $L^{10}:\mathcal{S}_{\mathbb{Z}}\to\widetilde{\mathcal{Z}}_{\mathbb{Z}}$ by
\[
e_1 e_0^{k_1-1}\cdots e_1 e_0^{k_r-1}
\longmapsto
(-1)^r \widetilde\zeta(k_1,\dots,k_r)
\qquad (k_r>1).
\]
Then $L^{10}$ is a homomorphism of $\mathbb{Z}$-algebras.}
\end{thm}
Let $\mathcal{S}_{\mathbb{Z}}^n$ denote the $\mathbb{Z}$-submodule of $\mathcal{S}_{\mathbb{Z}}$ spanned by all words of length $n$. Then we have
\[
L^{10}(\mathcal{S}_{\mathbb{Z}}^n)=\widetilde{\mathcal{Z}}_{\mathbb{Z}}^{n}.
\]
\section{Kontsevich graphs and their integrals}

In this section, we recall the definition of Kontsevich graphs and their associated integrals.
\begin{dfn}
Let $V$ be a finite set. A pair $G=(V,E)$ consisting of $V$ and a subset
$E\subset (V\times V)\setminus\{(x,x)\mid x\in V\}$ is called a
{\em directed simple graph}. When $G=(V,E)$ is a directed simple graph,
the elements of $V$ are called the {\em vertices} of $G$, and the elements
of $E$ are called the {\em edges} of $G$.
\end{dfn}

Let $G=(V,E)$ be a directed simple graph. For $x,y\in V$, if $(x,y)\in E$,
we depict this by
\[
\begin{tikzpicture}[baseline=(current bounding box.center)]
  \begin{scope}[decoration={
      markings,
      mark=at position 1.0 with {\arrow{>}}}
    ]
    \node (x) at (0,0) {$x$};
    \node (y) at (1,0) {$y$};
    \draw[postaction=decorate] (x) -- (y);
  \end{scope}
\end{tikzpicture}
\]
as usual.

\begin{ex}\label{gra}
Let $V=\{v_0,v_1,v_2,v_3\}$ and
$E=\{(v_0,v_1),(v_0,v_2),(v_0,v_3)\}$. Then the directed simple graph
$G=(V,E)$ is depicted as
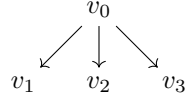
\[
\begin{tikzpicture}[baseline=(current bounding box.center)]
  \begin{scope}[decoration={
      markings,
      mark=at position 1.0 with {\arrow{>}}}
    ]
    \node (x) at (0,1) {$v_0$};
    \node (y) at (-1,0) {$v_1$};
    \node (z) at (0,0) {$v_2$};
    \node (w) at (1,0) {$v_3$};
    \draw[postaction=decorate] (x) -- (y);
    \draw[postaction=decorate] (x) -- (z);
    \draw[postaction=decorate] (x) -- (w);
  \end{scope}
\end{tikzpicture}
\]
\end{ex}

\begin{dfn}
Let $G=(V,E)$ be a directed simple graph. For a vertex $x\in V$, the
cardinality $\#\{(x,y)\in E\mid y\in V\}$ is called the {\em out-degree} of
$x$, and the cardinality $\#\{(y,x)\in E\mid y\in V\}$ is called the
{\em in-degree} of $x$.
\end{dfn}

\begin{ex}
For the directed simple graph $G=(V,E)$ in Example~\ref{gra}, the vertices
$v_1,v_2,v_3\in V$ have out-degree $0$ and in-degree $1$, while the vertex
$v_0\in V$ has out-degree $3$ and in-degree $0$.
\end{ex}

\begin{dfn}
Let $G=(V,E)$ be a directed simple graph. We say that $G$ is {\em connected}
if for any two distinct vertices $x,y\in V$, there exists a sequence of
vertices $\{a_i\}_{i=0}^n\subset V$ satisfying the following conditions:
\begin{itemize}
  \item[(1)] $a_0=x$ and $a_n=y$.
  \item[(2)] For every $1\le i\le n$, either $(a_{i-1},a_i)\in E$ or
  $(a_i,a_{i-1})\in E$ holds.
\end{itemize}
\end{dfn}

\begin{ex}
The directed simple graph in Example~\ref{gra} is connected.
\end{ex}

\begin{dfn}
For a positive integer $n$, set $[n]:=\{1,\dots,n\}$. Let $G=(V,E)$ be a
directed simple graph. A bijection $\varepsilon:E\to[\#E]$ is called an
{\em ordering} of the graph $G$. The pair $(G,\varepsilon)$ is then called
an {\em ordered directed simple graph}.
\end{dfn}

\begin{dfn}
For a positive integer $n\ge1$, a connected ordered directed simple graph
$(\Gamma,\varepsilon)$ is called a {\em Kontsevich graph of weight $n$} if
it satisfies the following conditions:
\begin{itemize}
  \item[(1)] $\Gamma$ has exactly $n+2$ vertices and $2n$ edges.
  \item[(2)] Among the $n+2$ vertices, exactly $n$ vertices have out-degree
  $2$, and the remaining $2$ vertices have out-degree $0$.
  \item[(3)] The two vertices of out-degree $0$ are labeled by $A$ and $B$
  without repetition.
\end{itemize}
Vertices of $\Gamma$ with out-degree $2$ are called {\em internal vertices},
and vertices with out-degree $0$ are called {\em external vertices}. A graph
with exactly two vertices and no edges is called a {\em Kontsevich graph of
weight $0$}.
\end{dfn}

\begin{ex}
For any choice of ordering, the directed simple graph $G$ in
Example~\ref{gra} is not a Kontsevich graph.
\end{ex}

\begin{ex}
Assigning an arbitrary ordering to any of the following graphs yields a
Kontsevich graph of weight $2$.
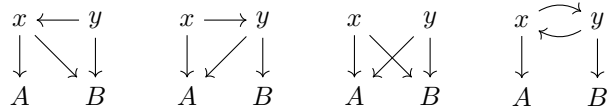
\[
    \begin{tikzpicture}[baseline=(current bounding box.center)]
		\begin{scope}[decoration={
				markings,
				mark=at position 1.0 with {\arrow{>}}}
			] 
			\node (u) at (0,-2) {$A$};
			\node (v) at (1,-2) {$B$};
			\node (y) at (0,-1){$x$};
			\node (z) at (1,-1){$y$};
			\draw[postaction=decorate] (z) -- (y);
			\draw[postaction=decorate] (z) -- (v);
			\draw[postaction=decorate] (y) -- (u);
			\draw[postaction=decorate] (y) -- (v);
		\end{scope}
	\end{tikzpicture}\qquad
	\begin{tikzpicture}[baseline=(current bounding box.center)]
		\begin{scope}[decoration={
				markings,
				mark=at position 1.0 with {\arrow{>}}}
			] 
			\node (u) at (0,-2) {$A$};
			\node (v) at (1,-2) {$B$};
			\node (y) at (0,-1){$x$};
			\node (z) at (1,-1){$y$};
			\draw[postaction=decorate] (y) -- (u);
			\draw[postaction=decorate] (y) -- (z);
			\draw[postaction=decorate] (z) -- (u);
			\draw[postaction=decorate] (z) -- (v);
		\end{scope}
	\end{tikzpicture}\qquad
    \begin{tikzpicture}[baseline=(current bounding box.center)]
		\begin{scope}[decoration={
				markings,
				mark=at position 1.0 with {\arrow{>}}}
			] 
			\node (u) at (0,-2) {$A$};
			\node (v) at (1,-2) {$B$};
			\node (x) at (0,-1){$x$};
			\node (y) at (1,-1){$y$};
			\draw[postaction=decorate] (x) -- (v);
			\draw[postaction=decorate] (x) -- (u);
			\draw[postaction=decorate] (y) -- (u);
			\draw[postaction=decorate] (y) -- (v);
		\end{scope}
	\end{tikzpicture}
    \qquad
    \begin{tikzpicture}[baseline=(current bounding box.center)]
				\begin{scope}[decoration={
						markings,
						mark=at position 1 with {\arrow{>}}}
					] 
					\node (U) at (0,-1) {$x$};
					\node (V) at (1,-1) {$y$};
					\node (A) at (0,-2) {$A$};
					\node (B) at (1,-2) {$B$};
					\draw[postaction=decorate] (U) to[bend left] (V);
					\draw[postaction=decorate] (V) to[bend left] (U);
					\draw[postaction=decorate] (U) -- (A);
					\draw[postaction=decorate] (V) -- (B);
				\end{scope}
			\end{tikzpicture}
	\]
Moreover, by convention, the external vertex labeled $A$ is drawn on the
left, and the external vertex labeled $B$ is drawn on the right.
\end{ex}

Let $n\ge1$. For a Kontsevich graph $(\Gamma,\varepsilon)$ of weight $n$ and
an element $\sigma$ of the symmetric group on $2n$ letters,
$(\Gamma,\sigma\circ\varepsilon)$ is again a Kontsevich graph.

\begin{dfn}
Let $\mathcal{K}_{\mathbb{Z}}$ denote the free $\mathbb{Z}$-module generated
by all Kontsevich graphs, modulo the relations
\[
(\Gamma,\varepsilon)=\sgn(\sigma)(\Gamma,\sigma\circ\varepsilon),
\]
where $\Gamma$ has weight at least $1$. Let $\mathcal{K}_{\mathbb{Z}}^n$
denote the submodule generated by equivalence classes of Kontsevich graphs
of weight $n$. Then
\[
\mathcal{K}_{\mathbb{Z}}=\bigoplus_{n=0}^{\infty}\mathcal{K}_{\mathbb{Z}}^n.
\]
Here $\mathcal{K}_{\mathbb{Z}}^0$ is the $\mathbb{Z}$-module generated by
Kontsevich graphs of weight $0$, which is canonically isomorphic to
$\mathbb{Z}$; hence we identify $\mathcal{K}_{\mathbb{Z}}^0$ with
$\mathbb{Z}$.
\end{dfn}

\begin{rem}
We often denote a Kontsevich graph $(\Gamma,\varepsilon)$, or its equivalence
class, simply by $\Gamma$.
\end{rem}

\begin{rem}
By identifying the external vertices of a Kontsevich graph with their
labels, we denote the set of external vertices by $R=\{A,B\}$.
\end{rem}

Let $n,m\ge1$. For the upper half-plane $\mathbb{H}$, set
\[
\mathrm{Conf}_n(\mathbb{H})
:=\{(z_1,\dots,z_n)\in\mathbb{H}^n \mid z_i\ne z_j\ (i\ne j)\},
\]
and
\[
\mathrm{Conf}_{m,+}(\mathbb{R})
:=\{(x_1,\dots,x_m)\in\mathbb{R}^m \mid x_1<\cdots<x_m\}.
\]
Let
\[
\mathrm{Aff}_{+}
:=\{x\mapsto ax+b \mid a\in\mathbb{R}_{>0},\, b\in\mathbb{R}\}
\]
denote the group of orientation-preserving affine transformations of
$\mathbb{H}$ (or $\mathbb{R}$). This group acts diagonally on
$\mathrm{Conf}_n(\mathbb{H})\times\mathrm{Conf}_{m,+}(\mathbb{R})$, and we set
\[
\mathfrak{C}_{n,m}
:=\bigl(\mathrm{Conf}_n(\mathbb{H})\times
\mathrm{Conf}_{m,+}(\mathbb{R})\bigr)/\mathrm{Aff}_{+}.
\]

\begin{rem}
The space $\mathfrak{C}_{n,m}$ is a real manifold of dimension $2n+m-2$; it
is orientable and connected.
\end{rem}

From now on, we use $(z_1,\dots,z_n,x_1,\dots,x_m)$ as the standard
coordinates on $\mathrm{Conf}_n(\mathbb{H})\times
\mathrm{Conf}_{m,+}(\mathbb{R})$.

\begin{dfn}
Let $\Gamma$ be a Kontsevich graph of weight $n$. Let $U$ denote the set of
internal vertices and $R=\{A,B\}$ the set of external vertices. Fix a
bijection $\upsilon:U\to[n]$. For $x\in U$ and $y\in U\sqcup R$ with $x\ne y$,
we define a $1$-form $\alpha_{x\to y}$ as follows:
\begin{itemize}
  \item[(1)] If $y\in U$, set
  \[
  \alpha_{x\to y}
  :=\frac{1}{2\pi i}
  d\log\!\left(
  \frac{z_{\upsilon(x)}-z_{\upsilon(y)}}
       {\overline{z_{\upsilon(x)}}-z_{\upsilon(y)}}
  \right)
  \in\Omega^1(\mathfrak{C}_{n,2}).
  \]
  \item[(2)] If $y=A$, set
  \[
  \alpha_{x\to A}
  :=\frac{1}{2\pi i}
  d\log\!\left(
  \frac{z_{\upsilon(x)}-x_1}
       {\overline{z_{\upsilon(x)}}-x_1}
  \right)
  \in\Omega^1(\mathfrak{C}_{n,2}).
  \]
  \item[(3)] If $y=B$, set
  \[
  \alpha_{x\to B}
  :=\frac{1}{2\pi i}
  d\log\!\left(
  \frac{z_{\upsilon(x)}-x_2}
       {\overline{z_{\upsilon(x)}}-x_2}
  \right)
  \in\Omega^1(\mathfrak{C}_{n,2}).
  \]
\end{itemize}
The form $\alpha_{x\to y}$ is called the {\em logarithmic propagator}.
\end{dfn}

\begin{rem}
For $i\ne j$, the expressions
$\frac{z_i-z_j}{\overline{z_i}-z_j}$,
$\frac{z_i-x_1}{\overline{z_i}-x_1}$, and
$\frac{z_i-x_2}{\overline{z_i}-x_2}$ are invariant under the action of
$\mathrm{Aff}_{+}$. Hence the logarithmic propagator can be regarded as a
$1$-form on $\mathfrak{C}_{n,2}$.
\end{rem}

Let $E$ denote the set of edges of a Kontsevich graph $(\Gamma,\varepsilon)$.
For $e_i\in E$ satisfying $\varepsilon(e_i)=i$, denote by $e_{i0}$ and
$e_{i1}$ the initial and terminal vertices of $e_i$, respectively. Set
\[
\alpha_\Gamma
:=\bigwedge_{i=1}^{2n}\alpha_{e_{i0}\to e_{i1}}
\in\Omega^{2n}(\mathfrak{C}_{n,2}),
\]
and define
\[
c_\Gamma
:=\int_{\mathfrak{C}_{n,2}}\alpha_\Gamma,
\]
which is called the {\em Kontsevich integral} of $\Gamma$, or simply the
{\em integral}. For the Kontsevich graph $\Gamma_0$ of weight $0$, we set
$c_{\Gamma_0}=1$.

\begin{rem}
The value of the Kontsevich integral of $\Gamma$ does not depend on the
choice of the bijection $\upsilon:U\to[n]$.
\end{rem}

It is known from~\cite{ARTW16} that $c_\Gamma$ converges absolutely. Moreover,
in~\cite{BPP} it is shown that for a Kontsevich graph $\Gamma$ of weight $n$,
we have $c_\Gamma\in\widetilde{\mathcal{Z}}_{\ZZ}^{\flat,n}$.

Let $(\Gamma_1,\varepsilon_1)$ and $(\Gamma_2,\varepsilon_2)$ be Kontsevich
graphs of weights $m$ and $n$, respectively. We define their {\em
join}, denoted by $(\Gamma_1*\Gamma_2,\varepsilon_1*\varepsilon_2)$,
to be the Kontsevich graph of weight $m+n$ constructed as follows.
\[
	\Gamma_1*\Gamma_2=\left(
	\begin{tikzpicture}[baseline=(current bounding box.center)]
		\begin{scope}[decoration={
				markings,
				mark=at position 1 with {\arrow{>}}}
			] 
            \node (G)at (0.5,1){$\Gamma_{1}$};
			\node (A) at (0,0) {$A$};
			\node (B) at (1,0) {$B$};
			\draw[postaction=decorate,dashed] (G) to (A);
			\draw[postaction=decorate,dashed] (G) to (B);
		\end{scope}
	\end{tikzpicture}
	\right)
	*
	\left(
	\begin{tikzpicture}[baseline=(current bounding box.center)]
		\begin{scope}[decoration={
				markings,
				mark=at position 1 with {\arrow{>}}}
			] 
            \node (G)at (0.5,1){$\Gamma_{2}$};
			\node (A) at (0,0) {$A$};
			\node (B) at (1,0) {$B$};
			\draw[postaction=decorate,dashed] (G) to (A);
			\draw[postaction=decorate,dashed] (G) to (B);
		\end{scope}
	\end{tikzpicture}
	\right)
	:=
	\begin{tikzpicture}[baseline=(current bounding box.center)]
		\begin{scope}[decoration={
				markings,
				mark=at position 1 with {\arrow{>}}}
			] 
			\node (A) at (0,0) {$A$};
			\node (B) at (1,0) {$B$};
			\node (G) at (0,1.5) {$\Gamma_1$};
			\node (H) at (1,1.5){$\Gamma_2$};
            \draw[postaction=decorate,dashed] (G) to (A);
			\draw[postaction=decorate,dashed] (G) to (B);
			\draw[postaction=decorate,dashed] (H) to (A);
			\draw[postaction=decorate,dashed] (H) to (B);
		\end{scope}
	\end{tikzpicture}.
	\]

The ordering $\varepsilon_1*\varepsilon_2$ is defined by
\[
\varepsilon_1*\varepsilon_2(e):=
\begin{cases}
\varepsilon_1(e), & e\in E(\Gamma_1),\\
2m+\varepsilon_2(e), & e\in E(\Gamma_2).
\end{cases}
\]
This yields a Kontsevich graph of weight $m+n$.

For equivalence classes
$[(\Gamma_1,\varepsilon_1)],[(\Gamma_2,\varepsilon_2)]\in\mathcal{K}_{\mathbb{Z}}$,
we define
\[
[(\Gamma_1,\varepsilon_1)]*[(\Gamma_2,\varepsilon_2)]
:=[(\Gamma_1*\Gamma_2,\varepsilon_1*\varepsilon_2)],
\]
and extend this operation linearly. Then
$(\mathcal{K}_{\mathbb{Z}},*)$ becomes a $\mathbb{Z}$-algebra.

Let $\sigma$ be the element of the symmetric group on $2m+2n$ letters given
by
\[
\sigma=
\begin{pmatrix}
1&\cdots&2m&2m+1&\cdots&2m+2n\\
2n+1&\cdots&2n+2m&1&\cdots&2n
\end{pmatrix}.
\]
Then $\sigma\circ(\varepsilon_1*\varepsilon_2)=\varepsilon_2*\varepsilon_1$,
and since $\sgn(\sigma)=1$, the algebra $(\mathcal{K}_{\mathbb{Z}},*)$ is
commutative. Moreover, by construction,
\[
\alpha_{\Gamma_1*\Gamma_2}
=\alpha_{\Gamma_1}\wedge\alpha_{\Gamma_2},
\]
and hence $c(\Gamma_1*\Gamma_2)=c(\Gamma_1)c(\Gamma_2)$. Therefore, by
extending the integral linearly, we obtain a homomorphism of $\mathbb{Z}$-algebras
\[
c:\mathcal{K}_{\mathbb{Z}}\to\widetilde{\mathcal{Z}}_{\ZZ}^{\flat,n}.
\]
\section{Ritland's method}
In this section, we describe the method introduced by Ritland in \cite{Rit}.
\begin{dfn}
  Let $\mathcal{F}:=\mathbb{Z}\langle \mu,\nu\rangle$ denote the noncommutative free $\mathbb{Z}$-algebra generated by $\mu$ and $\nu$.
\end{dfn}
For a Kontsevich graph $\Gamma$ of weight $n\ge1$, define $\mu(\Gamma)$ and $\nu(\Gamma)$ by
\[
\mu(\Gamma)=
\mu\left(
\begin{tikzpicture}[baseline=(current bounding box.center)]
  \begin{scope}[decoration={
      markings,
      mark=at position 1 with {\arrow{>}}}
    ] 
      \node (G)at (0.5,1) {$\Gamma$};
      \node (A) at (0,0) {$A$};
      \node (B) at (1,0) {$B$};
      \draw[postaction=decorate,dashed] (G) to (A);
      \draw[postaction=decorate,dashed] (G) to (B);
  \end{scope}
\end{tikzpicture}
\right) := 
\begin{tikzpicture}[baseline=(current bounding box.center)]
  \begin{scope}[decoration={
      markings,
      mark=at position 1 with {\arrow{>}}}
    ] 
      \node (G)at (0,1) {$\Gamma$};
      \node (A) at (0,-1) {$A$};
      \node (B) at (1,0) {$B'$};
      \node (C) at (1,-1) {$B$};
      \draw[postaction=decorate,dashed] (G) to (A);
      \draw[postaction=decorate,dashed] (G) to (B);
      \draw[postaction=decorate] (B) to (A);
      \draw[postaction=decorate] (B) to (C);
  \end{scope}
\end{tikzpicture}
\qquad
\nu(\Gamma)=\nu\left(
\begin{tikzpicture}[baseline=(current bounding box.center)]
  \begin{scope}[decoration={
      markings,
      mark=at position 1 with {\arrow{>}}}
    ] 
      \node (G)at (0.5,1) {$\Gamma$};
      \node (A) at (0,0) {$A$};
      \node (B) at (1,0) {$B$};
      \draw[postaction=decorate,dashed] (G) to (A);
      \draw[postaction=decorate,dashed] (G) to (B);
  \end{scope}
\end{tikzpicture}
\right) := 
\begin{tikzpicture}[baseline=(current bounding box.center)]
  \begin{scope}[decoration={
      markings,
      mark=at position 1 with {\arrow{>}}}
    ] 
      \node (G) at (1,1){$\Gamma$};
      \node (A) at (0,0) {$A'$};
      \node (B) at (1,-1) {$B$};
      \node (C) at (0,-1) {$A$};
      \draw[postaction=decorate,dashed] (G) to (A);
      \draw[postaction=decorate,dashed] (G) to (B);
      \draw[postaction=decorate] (A) to (B);
      \draw[postaction=decorate] (A) to (C);
  \end{scope}
\end{tikzpicture}.
\]
Moreover, we fix the orderings of $\mu(\Gamma)$ and $\nu(\Gamma)$ so that
\[
\alpha_{\mu(\Gamma)}=\alpha_{\Gamma}\wedge\alpha_{B'\to A}\wedge\alpha_{B'\to B},
\qquad
\alpha_{\nu(\Gamma)}=\alpha_{\Gamma}\wedge\alpha_{A'\to A}\wedge\alpha_{A'\to B}.
\]
For the Kontsevich graph $1$ of weight $0$, set $\mu(1)=\nu(1):=1$. In this way, $\mathcal{K}_{\mathbb{Z}}$ becomes an $\mathcal{F}$-module.

Similarly, for a nonempty word $\omega\in\mathcal{S}_{\ZZ}$, define
\[
\mu(\omega):=\omega e_0,
\qquad
\nu(\omega):=e_1\omega,
\]
and set $\mu(1)=\nu(1):=1$ for the empty word. Then $\mathcal{S}_{\ZZ}$ becomes an $\mathcal{F}$-module.
Below we recall the definitions of (undirected) simple graphs and rooted trees.

\begin{dfn}[Simple graph]
Let $V$ be a finite set. A pair $G=(V,E)$ consisting of $V$ and a family $E$ of subsets of two elements of $V$ is called an \emph{undirected simple graph}, or simply a \emph{simple graph}. Elements of $V$ are called the \emph{vertices} of $G$, and elements of $E$ are called the \emph{edges} of $G$.
\end{dfn}

Let $G=(V,E)$ be a simple graph. For $x,y\in V$, if $\{x,y\}\in E$, we depict this as
\[
\begin{tikzpicture}[baseline=(current bounding box.center)]
	\begin{scope}
		\node (x) at (0,0) {$x$};
		\node (y) at (1,0) {$y$};
		\draw[postaction=decorate] (x) -- (y);
	\end{scope}
\end{tikzpicture}
\]
as illustrated above.

\begin{ex}\label{mgra}
Let $V=\{v_0,v_1,v_2,v_3\}$ and $E=\{\{v_0,v_1\},\{v_0,v_2\},\{v_0,v_3\}\}$. Then the simple graph $G=(V,E)$ can be depicted as
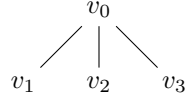
\[
\begin{tikzpicture}[baseline=(current bounding box.center)]
	\begin{scope}
		\node (x) at (0,1) {$v_0$};
		\node (y) at (-1,0) {$v_1$};
		\node (z) at (0,0) {$v_2$};
		\node (w) at (1,0) {$v_3$};
		\draw[postaction=decorate] (x) -- (y);
		\draw[postaction=decorate] (x) -- (z);
		\draw[postaction=decorate] (x) -- (w);
	\end{scope}
\end{tikzpicture}
\]
as shown above.
\end{ex}

As in the case of directed simple graphs, we define the degree and connectivity as follows.

\begin{dfn}[Degree]
Let $G=(V,E)$ be a simple graph. For $x\in V$, the number
\[
\#\{\{x,y\}\in E \mid y\in V\}
\]
is called the \emph{degree} of $x$.
\end{dfn}

\begin{ex}
In the simple graph of Example~\ref{mgra}, the degree of the vertex $v_0$ is $3$, and the degree of each of the other vertices is $1$.
\end{ex}

\begin{dfn}
Let $G=(V,E)$ be a simple graph. We say that $G$ is \emph{connected} if for any two distinct vertices $x,y\in V$, there exists a sequence of vertices $\{a_i\}_{i=0}^{n}\subset V$ satisfying the following conditions:
\begin{itemize}
	\item[(1)] $a_0=x$ and $a_n=y$.
	\item[(2)] For any $1\le i \le n$, we have $\{a_{i-1},a_i\}\in E$.
\end{itemize}
\end{dfn}

We next define cycles in simple graphs.

\begin{dfn}[Cycle]
Let $G=(V,E)$ be a simple graph. A sequence of vertices $\{a_i\}_{i=0}^{n}\subset V$ satisfying the following conditions is called a \emph{cycle} of $G$:
\begin{itemize}
	\item[(1)] $a_0=a_n$.
	\item[(2)] For any $1\le i \le n$, we have $\{a_{i-1},a_i\}\in E$.
	\item[(3)] For any $0\le i< j \le n$, we have $a_i\ne a_j$ unless $i=0$ and $j=n$.
\end{itemize}
\end{dfn}

We now define trees and rooted trees.

\begin{dfn}[Rooted tree]
A connected simple graph $T=(V,E)$ with no cycles is called a \emph{tree}. A vertex of a tree $T$ with degree $1$ is called a \emph{leaf}.  
For a tree $T=(V,E)$ and a vertex $x\in V$, the pair $(T,x)$ is called a \emph{rooted tree}, and $x$ is called the \emph{root} of $(T,x)$.  
A rooted tree $(T,x)$ is often denoted simply by $T$, omitting the root from the notation.
\end{dfn}

\begin{rem}
When drawing a rooted tree $T$, the root is represented by a $\bullet$.
\end{rem}

\begin{ex}
The graph in Example~\ref{mgra} is a tree. If we choose $v_0$ as the root, the corresponding rooted tree is depicted as
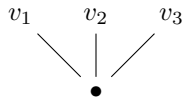
\[
\begin{tikzpicture}[baseline=(current bounding box.center)]
	\begin{scope}
		\node (x) at (0,-1) {$\bullet$};
		\node (y) at (-1,0) {$v_1$};
		\node (z) at (0,0) {$v_2$};
		\node (w) at (1,0) {$v_3$};
		\draw[postaction=decorate] (x) -- (y);
		\draw[postaction=decorate] (x) -- (z);
		\draw[postaction=decorate] (x) -- (w);
	\end{scope}
\end{tikzpicture}
\]
as shown above.  
By convention, the root is drawn at the lowest position.
\end{ex}
\begin{dfn}
For $n \ge 2$, let $\mathcal{T}_{\ZZ}^n$ denote the free $\ZZ$-module generated by all labeled rooted trees $T$ satisfying the following conditions.
\begin{itemize}
  \item[(1)] The degree of each vertex of $T$ is at most $3$.
  \item[(2)] Each leaf of $T$ is labeled by $e$.  
  Vertices of degree $2$ are labeled by either $\mu$ or $\nu$, and vertices of degree $3$ are labeled by $*$.
  \item[(3)] Define the weight of $T$ by $2\# e+\#\mu+\#\nu$.  
  Then
  \[
    2\# e+\#\mu+\#\nu = n,
  \]
  where $\#$ denotes the number of occurrences of the corresponding label in $T$.
\end{itemize}
We further set
\[
\mathcal{T}_{\ZZ}^0:=\ZZ,\qquad \mathcal{T}_{\ZZ}^1:=\{0\},
\]
and define
\[
\mathcal{T}_{\ZZ}:=\bigoplus_{n=0}^{\infty}\mathcal{T}_{\ZZ}^n .
\]
\end{dfn}

\begin{ex}
The rooted trees contained in $\mathcal{T}_{\ZZ}^4$ are given as follows.
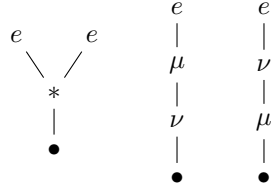
\begin{align*}
\begin{gathered}
\begin{tikzpicture}[baseline=(current bounding box.center)]
	\def\ysep{1.5}
	\node (A) at (-0.5,1*\ysep) {$e$};
	\node (B) at (0.5,1*\ysep) {$e$};
	\node (S) at (0,0.5*\ysep) {$*$};
	\node (R) at (0,0*\ysep) {$\bullet$};
	\draw (A)--(S)--(B);
	\draw (S)--(R);
\end{tikzpicture}
\end{gathered}
\qquad
\begin{gathered}
\begin{tikzpicture}[baseline=(current bounding box.center)]
	\def\ysep{1.5}
	\node (A) at (0,1.5*\ysep) {$e$};
	\node (B) at (0,1*\ysep) {$\mu$};
	\node (R) at (0,0.5*\ysep) {$\nu$};
	\node (W) at (0,0*\ysep) {$\bullet$};
	\draw (A)--(B);
	\draw (B)--(R);
	\draw (R)--(W);
\end{tikzpicture}
\end{gathered}
\qquad
\begin{gathered}
\begin{tikzpicture}[baseline=(current bounding box.center)]
	\def\ysep{1.5}
	\node (A) at (0,1.5*\ysep) {$e$};
	\node (B) at (0,1*\ysep) {$\nu$};
	\node (R) at (0,0.5*\ysep) {$\mu$};
	\node (W) at (0,0*\ysep) {$\bullet$};
	\draw (A)--(B);
	\draw (B)--(R);
	\draw (R)--(W);
\end{tikzpicture}
\end{gathered}
\end{align*}
\end{ex}

Moreover, we define
\begin{align*}
\begin{gathered}
\begin{tikzpicture}[baseline=(current bounding box.center)]
	\def\ysep{1}
	\node (A) at (0,1*\ysep) {$g_1$};
	\node (R) at (0,0*\ysep) {$\bullet$};
	\draw (A)--(R);
\end{tikzpicture}
\end{gathered}
\cdot
\begin{gathered}
\begin{tikzpicture}[baseline=(current bounding box.center)]
	\def\ysep{1}
	\node (A) at (0,1*\ysep) {$g_2$};
	\node (R) at (0,0*\ysep) {$\bullet$};
	\draw (A)--(R);
\end{tikzpicture}
\end{gathered}
:=
\begin{gathered}
\begin{tikzpicture}[baseline=(current bounding box.center)]
	\def\ysep{1.5}
	\node (A) at (-0.5,1*\ysep) {$g_1$};
	\node (B) at (0.5,1*\ysep) {$g_2$};
	\node (S) at (0,0.5*\ysep) {$*$};
	\node (R) at (0,0*\ysep) {$\bullet$};
	\draw (A)--(S)--(B);
	\draw (S)--(R);
\end{tikzpicture}
\end{gathered}
\qquad
\mu\!\left(
\begin{gathered}
\begin{tikzpicture}[baseline=(current bounding box.center)]
	\def\ysep{1}
	\node (A) at (0,1*\ysep) {$g$};
	\node (R) at (0,0*\ysep) {$\bullet$};
	\draw (A)--(R);
\end{tikzpicture}
\end{gathered}
\right)
:=
\begin{gathered}
\begin{tikzpicture}[baseline=(current bounding box.center)]
	\def\ysep{1.5}
	\node (A) at (0,1*\ysep) {$g$};
	\node (B) at (0,0.5*\ysep) {$\mu$};
	\node (R) at (0,0*\ysep) {$\bullet$};
	\draw (A)--(B);
	\draw (B)--(R);
\end{tikzpicture}
\end{gathered}
\qquad
\nu\!\left(
\begin{gathered}
\begin{tikzpicture}[baseline=(current bounding box.center)]
	\def\ysep{1}
	\node (A) at (0,1*\ysep) {$g$};
	\node (R) at (0,0*\ysep) {$\bullet$};
	\draw (A)--(R);
\end{tikzpicture}
\end{gathered}
\right)
:=
\begin{gathered}
\begin{tikzpicture}[baseline=(current bounding box.center)]
	\def\ysep{1.5}
	\node (A) at (0,1*\ysep) {$g$};
	\node (B) at (0,0.5*\ysep) {$\nu$};
	\node (R) at (0,0*\ysep) {$\bullet$};
	\draw (A)--(B);
	\draw (B)--(R);
\end{tikzpicture}
\end{gathered}
\end{align*}
and set $\mu(1):=1$, $\nu(1):=1$.
With these definitions, $\mathcal{T}_{\ZZ}$ becomes a $\ZZ$-algebra and an $\mathcal{F}$-module.
In addition, we denote the tree
\[
\begin{gathered}
\begin{tikzpicture}[baseline=(current bounding box.center)]
	\def\ysep{1}
	\node (A) at (0,1*\ysep) {$e$};
	\node (R) at (0,0*\ysep) {$\bullet$};
	\draw (A)--(R);
\end{tikzpicture}
\end{gathered}
\]
simply by $e$.
\begin{dfn}
Let
\[
G:\mathcal{T}_{\ZZ}\to\mathcal{K}_{\ZZ}
\]
be the map which is a homomorphism of $\mathcal{F}$-modules and also a homomorphism of $\ZZ$-algebras.  
We define $G$ by
\[
G(e)=
\begin{tikzpicture}[baseline=(current bounding box.center)]
	\begin{scope}[decoration={
			markings,
			mark=at position 1 with {\arrow{>}}}
		] 
		\node (U) at (0,-1) {$x$};
		\node (V) at (1,-1) {$y$};
		\node (A) at (0,-2) {$A$};
		\node (B) at (1,-2) {$B$};
		\draw[postaction=decorate] (U) to[bend left] (V);
		\draw[postaction=decorate] (V) to[bend left] (U);
		\draw[postaction=decorate] (U) -- (A);
		\draw[postaction=decorate] (V) -- (B);
	\end{scope}
\end{tikzpicture}
\]
where the ordering on the right-hand side is chosen so that
\[
\alpha_{\Gamma}
=
\alpha_{x\to y}\wedge
\alpha_{y\to x}\wedge
\alpha_{x\to A}\wedge
\alpha_{y\to B}.
\]
\end{dfn}

\begin{dfn}
Let
\[
I:\mathcal{T}_{\ZZ}\to \mathcal{S}_{\ZZ}
\]
be the map which is a homomorphism of $\mathcal{F}$-modules and also a homomorphism of $\ZZ$-algebras.  
We define $I$ by
\[
I(e)=e_1e_0.
\]
\end{dfn}
  It was shown in \cite{Rit} that the maps $G$ and $I$ are uniquely determined by these conditions.
Under these setups, Ritland proved the following proposition in \cite{Rit}.

\begin{prop}[{\cite[Proposition 4.13]{Rit}}]\label{com}
\em{For any $T\in\mathcal{T}_{\ZZ}^n$, we have
\[
c(G(T))=L^{10}(I(T)).
\]}
\end{prop}
As a consequence, the computation of integrals in $\mathcal{K}_{\ZZ}$ can be reduced to purely algebraic computations in $\mathcal{S}_{\ZZ}$.  
Ritland proved the following two theorems.

\begin{thm}[{\cite[Theorem 5.8]{Rit}}]\label{wg}
\em{For any $n\ge 0$, we have
\[
c(\mathcal{K}_{\QQ}^n)={\widetilde{\mathcal{Z}}^{\flat,n}_{\QQ}}.
\]
where, $\mathcal{K}_{\QQ}^n:=\mathcal{K}_{\ZZ}^n\otimes \QQ\,,\quad{\widetilde{\mathcal{Z}}^{\flat,n}_{\QQ}}:={\widetilde{\mathcal{Z}}^{\flat,n}_{\ZZ}}\otimes\QQ$.}
\end{thm}

\begin{thm}[{\cite[Theorem 5.9]{Rit}}]\label{ag}
\em{The integration map
\[
c:\mathcal{K}_{\ZZ}\to{\widetilde{\mathcal{Z}}^{\flat}_{\ZZ}}
\]
is surjective. That is,
\[
c(\mathcal{K}_{\ZZ})={\widetilde{\mathcal{Z}}^{\flat}_{\ZZ}}.
\]}
\end{thm}
In \cite{Rit}, Theorem~\ref{wg} is proved by means of Corollary~\ref{qdeseisei} stated below. In this paper, we provide a refinement of these results in Section~6.
\section{An explicit construction of Kontsevich graphs}
In this section, we prove Theorem~\ref{rafunathm}. We begin by briefly explaining the motivation for the theorem, and then introduce the tools required for the proof. Finally, we proceed to the proof itself.

By Proposition~\ref{com}, for any word $\omega\in \mathcal{S}_{\ZZ}^n$, if there exists
$T\in\mathcal{T}_{\ZZ}^n$ such that $I(T)=\omega$, then we obtain
\[
L^{10}(\omega)=c(G(T)),
\]
and hence
\[
L^{10}(\omega)\in c(\mathcal{K}_{\ZZ}^n).
\]
In \cite{Rit}, Ritland showed that such an element $T$ indeed exists in
$\mathcal{T}_{\QQ}^n$.

However, the map $I$ is not injective. In fact, for
\begin{align*}
T_1&=\mu(e\cdot\mu^2(e))-2\mu^{3}(e\cdot\mu(e))+\mu^{4}(e\cdot e),\\
T_2&=\frac{1}{2}\bigl(\mu^{2}(e)\cdot \mu^{2}(e)\bigr)
      -4\mu^{3}(e\cdot\mu(e))-3\mu^{4}(e\cdot e),
\end{align*}
we have
\[
I(T_1)=I(T_2).
\]
Here, $T_1\in \mathcal{T}_{\ZZ}^8$, whereas
$T_2\in \mathcal{T}_{\QQ}^8\setminus\mathcal{T}_{\ZZ}^8$.
Consequently, from $T_1$ we obtain
\[
\widetilde\zeta(4,4)\in c(\mathcal{K}_{\ZZ}^8),
\]
while from $T_2$ we cannot obtain
\[
\widetilde\zeta(4,4)\in c(\mathcal{K}_{\ZZ}^8).
\]
Thus, by explicitly constructing a more suitable element
$T\in\mathcal{T}_{\QQ}^n$
for a given word $\omega\in\mathcal{S}_{\ZZ}^{n}$, one can obtain more refined information on
$c(\mathcal{K}_{\ZZ}^n)$ and $c(\mathcal{K}_{\ZZ})$.
Theorem~\ref{rafunathm} is obtained by providing an explicit construction of an element
$T\in\mathcal{T}_{\ZZ}^n$
for a specific word, and will subsequently yield Corollary~\ref{mzge4}.

Let $T_1,T_2\in \mathcal{T}_{\ZZ}$. 
We define
\[
H(T_1,T_2)
:=T_1\cdot \mu^2(T_2)
-2\,\mu\bigl(T_1\cdot \mu(T_2)\bigr)
+\mu^2(T_1\cdot T_2).
\]
By extending linearly, this defines a bilinear map
\[
H:\mathcal{T}_{\ZZ}\times \mathcal{T}_{\ZZ}\to \mathcal{T}_{\ZZ}.
\]
For $k\ge2$, we write
\[
H_k(T):=H\bigl(e,\mu^{k-2}(T)\bigr).
\]
Similarly, we define $H$ and $H_k$ on $\mathcal{S}_{\ZZ}$ and $\mathcal{K}_{\ZZ}$.

By the definition of $H$ and Proposition~\ref{com}, the maps $G$ and $H$, as well as $I$ and $H$, commute with each other.

\begin{rem}\label{Hgra}
For Kontsevich graphs $\Gamma_1,\Gamma_2$, the element $H(\Gamma_1,\Gamma_2)$ can be depicted as follows:
\[
H(\Gamma_1,\Gamma_2)
=
\left(
\begin{tikzpicture}[baseline=(current bounding box.center)]
	\begin{scope}[decoration={
			markings,
			mark=at position 0.6 with {\arrow{>}}}
		] 
		\node (A) at (0,0) {$A$};
		\node (B) at (2,0) {$B$};
		\node (G) at (0,2.4){$\Gamma_1$};
        \node (F) at (2,2.4){$\Gamma_2$};
        \node (x) at (2,1.6){$\bullet$};
        \node (y) at (2,0.8){$\bullet$};
		\draw[postaction=decorate,dashed] (G) to (A);
		\draw[postaction=decorate,dashed] (G) to (B);
        \draw[postaction=decorate,dashed] (F) to (A);
		\draw[postaction=decorate,dashed] (F) to (x);
        \draw[postaction=decorate] (x) to (A);
		\draw[postaction=decorate] (x) to (y);
        \draw[postaction=decorate] (y) to (A);
		\draw[postaction=decorate] (y) to (B);
	\end{scope}
\end{tikzpicture}
\right)
-2
\left(
\begin{tikzpicture}[baseline=(current bounding box.center)]
	\begin{scope}[decoration={
			markings,
			mark=at position 0.5 with {\arrow{>}}}
		] 
		\node (A) at (0,0) {$A$};
		\node (B) at (2,0) {$B$};
		\node (G) at (0,2.4){$\Gamma_1$};
        \node (F) at (2,2.4){$\Gamma_2$};
        \node (x) at (2,1.6){$\bullet$};
        \node (y) at (2,0.8){$\bullet$};
		\draw[postaction=decorate,dashed] (G) to (A);
		\draw[postaction=decorate,dashed] (G) to (y);
        \draw[postaction=decorate,dashed] (F) to (A);
		\draw[postaction=decorate,dashed] (F) to (x);
        \draw[postaction=decorate] (x) to (A);
		\draw[postaction=decorate] (x) to (y);
        \draw[postaction=decorate] (y) to (A);
		\draw[postaction=decorate] (y) to (B);
	\end{scope}
\end{tikzpicture}
\right)
+
\left(
\begin{tikzpicture}[baseline=(current bounding box.center)]
	\begin{scope}[decoration={
			markings,
			mark=at position 0.6 with {\arrow{>}}}
		] 
		\node (A) at (0,0) {$A$};
		\node (B) at (2,0) {$B$};
		\node (G) at (0,2.4){$\Gamma_1$};
        \node (F) at (2,2.4){$\Gamma_2$};
        \node (x) at (2,1.6){$\bullet$};
        \node (y) at (2,0.8){$\bullet$};
		\draw[postaction=decorate,dashed] (G) to (A);
		\draw[postaction=decorate,dashed] (G) to (x);
        \draw[postaction=decorate,dashed] (F) to (A);
		\draw[postaction=decorate,dashed] (F) to (x);
        \draw[postaction=decorate] (x) to (A);
		\draw[postaction=decorate] (x) to (y);
        \draw[postaction=decorate] (y) to (A);
		\draw[postaction=decorate] (y) to (B);
	\end{scope}
\end{tikzpicture}
\right).
\]
\end{rem}

For the operator $H$ on $\mathcal{S}_{\ZZ}$, the following lemma holds.

\begin{lem}\label{H}
\em{Let $\omega=\omega' b e_0$ with $b\in\{e_0,e_1\}$. Then
\[
H(\omega,v)=(\omega'\shuffle \mu^2(v))\, b e_0
\]
holds. Here $\omega'$ denotes an arbitrary word, possibly the empty word.}
\end{lem}
\begin{proof}
By the definition of the shuffle product, we have
\begin{align*}
\omega \shuffle \mu^{2}(v)
&= (\omega' b \shuffle \mu^{2}(v)) e_0 + \mu(\omega \shuffle \mu(v)) \\
&= (\omega' \shuffle \mu^{2}(v)) b e_0
  + (\omega' b \shuffle \mu(v)) e_0 e_0
  + \mu(\omega \shuffle \mu(v)).
\end{align*}
On the other hand, we have
\[
(\omega \shuffle \mu(v)) e_0
= (\omega' b \shuffle \mu(v)) e_0 e_0
  + \mu^{2}(\omega \shuffle v).
\]
Hence,
\[
(\omega' b \shuffle \mu(v)) e_0 e_0
= \mu(\omega \shuffle \mu(v)) - \mu^{2}(\omega \shuffle v).
\]
Substituting this into the previous expression and simplifying, we obtain
\[
H(\omega,v) = (\omega' \shuffle \mu^{2}(v)) b e_0,
\]
as desired.
\end{proof}

By Lemma~\ref{H}, we immediately obtain the following result.

\begin{prop}\label{Hk}
\em{For any $v\in \mathcal{S}_{\ZZ}$, we have
\[
H_k(v)=\mu^{k}(v)e_1 e_0.
\]}
\end{prop}
From Proposition~\ref{Hk} and Proposition~\ref{com}, we can deduce Theorem~1.1.
\begin{thm}\label{ge4}
\emph{Let $k_1,\dots,k_r \ge 2$ be natural numbers. Define $H_{(k_1,\dots,k_r)} := H_{k_r} \circ \cdots \circ H_{k_1}$. Then
\[
c\bigl(H_{(k_1,\dots,k_r)}(G(e))\bigr)
= (-1)^{r+1}\widetilde\zeta(k_1+2,\dots,k_r+2,2).
\]
}
\end{thm}

\begin{proof}
By Propositions~\ref{Hk} and~\ref{com}, we compute
\begin{align*}
c\bigl((H_{k_r}\circ\cdots\circ H_{k_1})(G(e))\bigr)
&= c\circ G\bigl((H_{k_r}\circ\cdots\circ H_{k_1})(e)\bigr) \\
&= L^{10}\circ I\bigl((H_{k_r}\circ\cdots\circ H_{k_1})(e)\bigr) \\
&= L^{10}\bigl((H_{k_r}\circ\cdots\circ H_{k_1})(e_1e_0)\bigr) \\
&= L^{10}\bigl(e_1 e_0^{k_1+1}\cdots e_0^{k_r+1} e_1 e_0\bigr) \\
&= (-1)^{r+1}\widetilde{\zeta}(k_1+2,\dots,k_r+2,2).
\end{align*}
This completes the proof.
\end{proof}

Moreover, by the definition of $H_k$, we have
\[
H_k(\mathcal{K}_{\ZZ}^n)\subset \mathcal{K}_{\ZZ}^{n+k+2}.
\]
Therefore, Theorem~\ref{ge4} immediately implies the following corollary.

\begin{cor}\em{\label{mzge4}
Let $k_r\ge 2$ and $k_1,\dots,k_{r-1}\ge 4$ be natural numbers. Then
\[
\widetilde\zeta(k_1,\dots,k_r)
\in c\bigl(\mathcal{K}_{\ZZ}^{k_1+\cdots+k_r}\bigr).
\]}
\end{cor}

\begin{proof}
By Theorem~\ref{ge4}, for natural numbers $k_r\ge 2$ and $k_1,\dots,k_{r-1}\ge 4$, we have
\[
c\bigl(\mu^{k_r-2}\bigl((H_{k_{r-1}-2}\circ\cdots\circ H_{k_1-2})(G(e))\bigr)\bigr)
= \frac{(-1)^{r}}{(2\pi i)^{k_1+\cdots+k_r}}
  \zeta(k_1,\dots,k_r).
\]
Since
\[
\mu^{k_r-2}\bigl((H_{k_{r-1}-2}\circ\cdots\circ H_{k_1-2})(G(e))\bigr)
\in \mathcal{K}_{\ZZ}^{k_1+\cdots+k_r},
\]
the claim follows.
\end{proof}
 The Kontsevich graph $(H_{k_r} \circ \cdots \circ H_{k_1})(e)$ appearing in
Theorem~\ref{ge4} can be described explicitly.
\begin{ex}The following is the Kontsevich graph $H_{5}(G(e))$ whose integral is $\widetilde{\zeta}(7,2)$.
      	\begin{align*}
			\left(
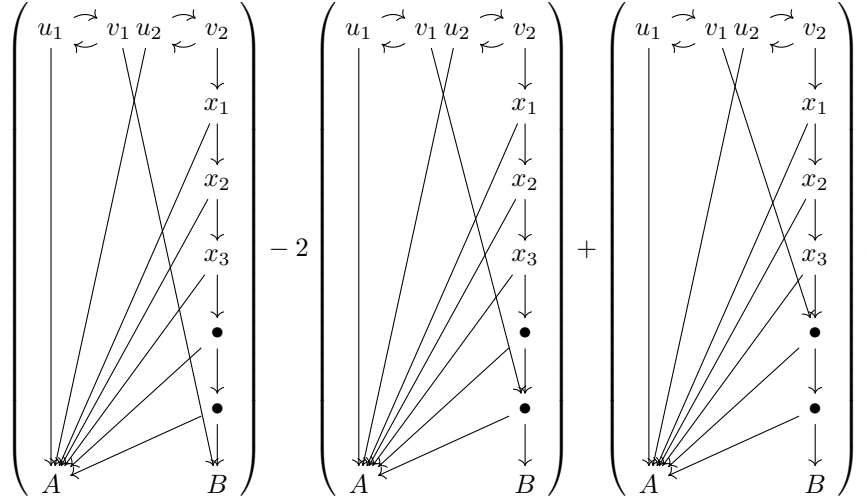
				\begin{tikzpicture}[baseline=(current bounding box.center)]
					\begin{scope}[decoration={
							markings,
							mark=at position 1 with {\arrow{>}}}
						] 
						\node (U) at (0,0) {$u_1$};
						\node (V) at (0.9,0) {$v_1$};
						\node (A) at (0,-6) {$A$};
						\node (B) at (2.2,-6) {$B$};
						\node (X) at (1.3,0) {$u_2$};
						\node (Y) at (2.2,0) {$v_2$};  
						\node (x_1) at (2.2,-1) {$x_1$};     
						\node (p_2) at (2.2,-2.0) {$x_2$}; 
						\node (x_2) at (2.2,-3) {$x_3$};
						\node (x_3) at (2.2,-4) {$\bullet$}; 
						\node (x_4) at (2.2,-5) {$\bullet$};     
						\draw[postaction=decorate] (U) to[bend left] (V);
						\draw[postaction=decorate] (V) to[bend left] (U);
						\draw[postaction=decorate] (X) to[bend left] (Y);
						\draw[postaction=decorate] (Y) to[bend left] (X);				
						\draw[postaction=decorate] (X) -- (A);				                    \draw[postaction=decorate] (U) -- (A);
						\draw[postaction=decorate] (Y) -- (x_1);                   
						\draw[postaction=decorate] (x_2) -- (x_3);
						\draw[postaction=decorate] (x_3) -- (x_4);
						\draw[postaction=decorate] (x_1) -- (p_2);
                        \draw[postaction=decorate] (p_2) -- (A);
						\draw[postaction=decorate] (p_2) -- (x_2);
						\draw[postaction=decorate] (x_2) -- (A);
						\draw[postaction=decorate] (x_3) -- (A);
						\draw[postaction=decorate] (x_4) -- (B);
						\draw[postaction=decorate] (x_1) -- (A);
						\draw[postaction=decorate] (V) -- (B);
                        \draw[postaction=decorate] (x_4) -- (A);
					\end{scope}
				\end{tikzpicture}
				\right)
                -2\left(
				\begin{tikzpicture}[baseline=(current bounding box.center)]
					\begin{scope}[decoration={
							markings,
							mark=at position 1 with {\arrow{>}}}
						] 
						\node (U) at (0,0) {$u_1$};
						\node (V) at (0.9,0) {$v_1$};
						\node (A) at (0,-6) {$A$};
						\node (B) at (2.2,-6) {$B$};
						\node (X) at (1.3,0) {$u_2$};
						\node (Y) at (2.2,0) {$v_2$};  
						\node (x_1) at (2.2,-1) {$x_1$};     
						\node (p_2) at (2.2,-2.0) {$x_2$}; 
						\node (x_2) at (2.2,-3) {$x_3$};
						\node (x_3) at (2.2,-4) {$\bullet$}; 
						\node (x_4) at (2.2,-5) {$\bullet$};     
						\draw[postaction=decorate] (U) to[bend left] (V);
						\draw[postaction=decorate] (V) to[bend left] (U);
						\draw[postaction=decorate] (X) to[bend left] (Y);
						\draw[postaction=decorate] (Y) to[bend left] (X);				
						\draw[postaction=decorate] (X) -- (A);				                    \draw[postaction=decorate] (U) -- (A);
						\draw[postaction=decorate] (Y) -- (x_1);                   
						\draw[postaction=decorate] (x_2) -- (x_3);
						\draw[postaction=decorate] (x_3) -- (x_4);
						\draw[postaction=decorate] (x_1) -- (p_2);
                        \draw[postaction=decorate] (p_2) -- (A);
						\draw[postaction=decorate] (p_2) -- (x_2);
						\draw[postaction=decorate] (x_2) -- (A);
						\draw[postaction=decorate] (x_3) -- (A);
						\draw[postaction=decorate] (x_4) -- (B);
						\draw[postaction=decorate] (x_1) -- (A);
						\draw[postaction=decorate] (V) -- (x_4);
                        \draw[postaction=decorate] (x_4) -- (A);
					\end{scope}
				\end{tikzpicture}
				\right)
                +\left(
				\begin{tikzpicture}[baseline=(current bounding box.center)]
					\begin{scope}[decoration={
							markings,
							mark=at position 1 with {\arrow{>}}}
						] 
						\node (U) at (0,0) {$u_1$};
						\node (V) at (0.9,0) {$v_1$};
						\node (A) at (0,-6) {$A$};
						\node (B) at (2.2,-6) {$B$};
						\node (X) at (1.3,0) {$u_2$};
						\node (Y) at (2.2,0) {$v_2$};  
						\node (x_1) at (2.2,-1) {$x_1$};     
						\node (p_2) at (2.2,-2.0) {$x_2$}; 
						\node (x_2) at (2.2,-3) {$x_3$};
						\node (x_3) at (2.2,-4) {$\bullet$}; 
						\node (x_4) at (2.2,-5) {$\bullet$};     
						\draw[postaction=decorate] (U) to[bend left] (V);
						\draw[postaction=decorate] (V) to[bend left] (U);
						\draw[postaction=decorate] (X) to[bend left] (Y);
						\draw[postaction=decorate] (Y) to[bend left] (X);				
						\draw[postaction=decorate] (X) -- (A);				                    \draw[postaction=decorate] (U) -- (A);
						\draw[postaction=decorate] (Y) -- (x_1);                   
						\draw[postaction=decorate] (x_2) -- (x_3);
						\draw[postaction=decorate] (x_3) -- (x_4);
						\draw[postaction=decorate] (x_1) -- (p_2);
                        \draw[postaction=decorate] (p_2) -- (A);
						\draw[postaction=decorate] (p_2) -- (x_2);
						\draw[postaction=decorate] (x_2) -- (A);
						\draw[postaction=decorate] (x_3) -- (A);
						\draw[postaction=decorate] (x_4) -- (B);
						\draw[postaction=decorate] (x_1) -- (A);
						\draw[postaction=decorate] (V) -- (x_3);
                        \draw[postaction=decorate] (x_4) -- (A);
					\end{scope}
				\end{tikzpicture}
				\right).
			\end{align*}
\end{ex}
\section{Proof of Theorem 1.2}
In this section, we give a proof of Theorem~1.2. For this purpose, we first recall the definition of Lyndon words and the decomposition of words based on them.
\begin{dfn}[Lyndon word]
We equip the set of all words in the alphabet $\{e_0,e_1\}$ with the lexicographic order induced by $e_1 < e_0$. A word $\omega$ is called a \textbf{Lyndon word} if for every factorization $\omega = x \cdot y$ with $x,y \neq \emptyset$, one has $x < y$.
\end{dfn}

\begin{rem}
Any Lyndon word of length at least $2$ belongs to $\mathcal{S}_{\ZZ}$.
\end{rem}

Any word admits a decomposition in terms of Lyndon words as follows.

\begin{prop}[\cite{VM},{\cite[Proposition 5.6]{Rit}}]\label{Lyshu}
\emph{Let $\omega \in \mathcal{S}_{\ZZ}$ be a word of length $m$. Then $\lfloor \tfrac{m}{2} \rfloor! \cdot \omega$ can be expressed as a $\ZZ$-linear combination of shuffle products of Lyndon words in $\mathcal{S}_{\ZZ}$.}
\end{prop}

Moreover, the following lemma holds for Lyndon words.

\begin{lem}[{\cite[Proposition 5.7]{Rit}}]\label{L3}
\emph{Any Lyndon word of length at least $3$ either ends with $e_0 e_0$ or begins with $e_1 e_1$.}
\end{lem}
        By Proposition~\ref{Lyshu} and Lemma~\ref{L3}, we obtain the following proposition.

\begin{prop}\label{nijoudeosaeru}
\emph{For any word $\omega \in \mathcal{S}_{\ZZ}^m$ of length $m$, one has
\[
\left\lfloor \frac{m(m+1)}{4} \right\rfloor! \cdot \omega \in I(\mathcal{T}_{\ZZ}^m).
\]}
\end{prop}

\begin{proof}
We proof by induction on the length of $\omega$. The case $|\omega|=2$ is clear. Assume $|\omega|=m>2$.

By Proposition~\ref{Lyshu}, the decomposition of $\omega$ into Lyndon words can be written as
\[
\left\lfloor \frac{m}{2} \right\rfloor! \cdot \omega
= \sum_i c_i \bigl(l_{i_1} \shuffle \cdots \shuffle l_{i_{r_i}}\bigr)
\qquad (c_i \in \ZZ),
\]
where $n_j := |l_{i_j}| \ge 2$ and $m = \sum_{1 \le j \le r_i} n_j$.

If $n_j = 2$ for all $j$, then $l_{i_j} = e_1 e_0$ and hence
\[
l_{i_1} \shuffle \cdots \shuffle l_{i_{r_i}}
= e_1 e_0 \shuffle \cdots \shuffle e_1 e_0
\in I(\mathcal{T}_{\ZZ}^m).
\]

Suppose that there exists $j$ such that $n_j > 2$. By Lemma~\ref{L3}, we can write
\[
l_{i_j} = \mu(l_{i_j}') \quad \text{or} \quad l_{i_j} = \nu(l_{i_j}'),
\]
where $|l_{i_j}'| = n_j - 1 < m$. By the induction hypothesis,
\[
\left\lfloor \frac{n_j(n_j - 1)}{4} \right\rfloor! \cdot l_{i_j}'
\in I(\mathcal{T}_{\ZZ}^{n_j - 1}),
\]
and hence
\[
\left\lfloor \frac{n_j(n_j - 1)}{4} \right\rfloor! \cdot l_{i_j}
\in I(\mathcal{T}_{\ZZ}^{n_j}).
\]

Moreover, by the induction hypothesis, for any $1 \le k \le r_i$,
\[
\left\lfloor \frac{n_k(n_k + 1)}{4} \right\rfloor! \cdot l_{i_k}
\in I(\mathcal{T}_{\ZZ}^{n_k}).
\]
Therefore,
\[
\left( \prod_{k \ne j} \left\lfloor \frac{n_k(n_k + 1)}{4} \right\rfloor! \right)
\cdot
\left\lfloor \frac{n_j(n_j - 1)}{4} \right\rfloor!
\cdot
\bigl(l_{i_1} \shuffle \cdots \shuffle l_{i_{r_i}}\bigr)
\in I(\mathcal{T}_{\ZZ}^m).
\]

Since $\sum_{k \ne j} n_k + (n_j - 1) \le m - 1$, we have
\[
\sum_{k \ne j} \left\lfloor \frac{n_k(n_k + 1)}{4} \right\rfloor
+
\left\lfloor \frac{n_j(n_j - 1)}{4} \right\rfloor
\le
\left\lfloor \frac{m(m - 1)}{4} \right\rfloor.
\]
It follows that
\[
\left\lfloor \frac{m(m - 1)}{4} \right\rfloor!
\cdot
\bigl(l_{i_1} \shuffle \cdots \shuffle l_{i_{r_i}}\bigr)
\in I(\mathcal{T}_{\ZZ}^m).
\]

Consequently,
\[
\left\lfloor \frac{m(m - 1)}{4} \right\rfloor!
\cdot
\left\lfloor \frac{m}{2} \right\rfloor!
\cdot \omega
\in I(\mathcal{T}_{\ZZ}^m),
\]
which implies in particular that
\[
\left\lfloor \frac{m(m + 1)}{4} \right\rfloor!
\cdot \omega
\in I(\mathcal{T}_{\ZZ}^m).
\]
\end{proof}
As an immediate consequence of the above proposition, we obtain the following corollary.

\begin{cor}[{\cite{Rit}}]\label{qdeseisei}
\emph{For any $n \ge 0$, the map
\[
I : \mathcal{T}_{\QQ}^n \to \mathcal{S}_{\QQ}^n
\]
is surjective.}
\end{cor}

By Corollary~\ref{qdeseisei} together with Proposition~\ref{com}, Theorem~\ref{wg} follows. Moreover, the following lemma holds for the image of $c$.

\begin{lem}\label{1/2}
\emph{Let $n \ge 1$ be an integer. For $z \in \widetilde{\mathcal{Z}}^{\flat}_{\ZZ}$, if $z \in c(\mathcal{K}_{\ZZ}^n)$, then
\[
\frac{1}{2}z \in c(\mathcal{K}_{\ZZ}^{n+1}).
\]
In particular, one has $c(\mathcal{K}_{\ZZ}^n) \subset c(\mathcal{K}_{\ZZ}^{n+1})$.}
\end{lem}

\begin{proof}
For $z \in \widetilde{\mathcal{Z}}^{\flat}_{\ZZ}$, suppose that there exists $\Gamma \in \mathcal{K}_{\ZZ}^n$ such that $c(\Gamma) = z$.
Consider the Kontsevich graph of weight $1$
\[
\Gamma_{w}:=\begin{tikzpicture}[baseline=(current bounding box.center)]
\begin{scope}[decoration={
markings,
mark=at position 1 with {\arrow{>}}}
] 
\node (X) at (0,0) {$x$};
\node (U) at (-0.5,-1) {$A$};
\node (V) at (0.5,-1) {$B$};
\draw[postaction=decorate] (X) -- (U);
\draw[postaction=decorate] (X) -- (V);
\end{scope}
\end{tikzpicture}
\]
It is shown that
\[
c(\Gamma_w)
= \int_x \alpha_{\Gamma_w}
= \int_x \alpha_{x \to A} \wedge \alpha_{x \to B}
= -\frac{1}{2},
\]
(see {\cite[Lemma 5.3]{BPP}}).
Since $\Gamma * \Gamma_w \in \mathcal{K}_{\ZZ}^{n+1}$, it follows that
\[
-c(\Gamma * \Gamma_w) = \frac{1}{2}z \in c(\mathcal{K}_{\ZZ}^{n+1}).
\]
The inclusion $c(\mathcal{K}_{\ZZ}^n) \subset c(\mathcal{K}_{\ZZ}^{n+1})$ follows as well.
\end{proof}

Now, by Corollary~\ref{mzge4}, we have
\[
\widetilde{\zeta}(\{4\}_n) := \widetilde{\zeta}(4,\dots,4) \in c(\mathcal{K}_{\ZZ}^{4n}).
\]
For $\widetilde{\zeta}(\{4\}_n)$, the following explicit formula is known.

\begin{prop}[{\cite{BBB}}]\label{444}
\emph{For any integer $n \ge 1$, one has
\[
\widetilde{\zeta}(\{4\}_n) = \frac{1}{2^{2n-1}(4n+2)!}.
\]}
\end{prop}

From Proposition~\ref{nijoudeosaeru}, Lemma~\ref{1/2}, and Proposition~\ref{444}, we obtain the following Theorem 1.2.
\begin{proof}[Proof of Theorem 1.2]
First, we show that for every $m\ge 2$,
\begin{equation}\label{tildegafukumareru}
  \widetilde{\mathcal{Z}}_{\ZZ}^m \subset c(\mathcal{K}_{\ZZ}^{m+4N(m)}).
\end{equation}

By Proposition~\ref{444}, for every $n\ge 1$,
\[
\frac{1}{(4n+2)!} \in c(\mathcal{K}_{\ZZ}^{4n}).
\]
Moreover, by Propositions~\ref{com} and~\ref{nijoudeosaeru}, for every word
\[
\omega = e_1 e_0^{k_1-1} \cdots e_1 e_0^{k_r-1} \in \mathcal{S}_{\ZZ}
\]
of length $m\ge2$, we have
\[
\left\lfloor \frac{m(m+1)}{4} \right\rfloor!
\cdot \widetilde{\zeta}(k_1,\dots,k_r)
\in L^{10}\bigl(I(\mathcal{T}_{\ZZ}^m)\bigr)
\subset c(\mathcal{K}_{\ZZ}^m).
\]
Therefore, if $n$ is chosen so that
\[
\left\lfloor \frac{m(m+1)}{4} \right\rfloor \le 4n+2,
\]
then
\[
\widetilde{\zeta}(k_1,\dots,k_r)
\in c(\mathcal{K}_{\ZZ}^{m+4n}).
\]

Recall that
\[
N(m):=
\begin{cases}
\left\lceil \dfrac{m^2+m-8}{16} \right\rceil
& \text{if $m \equiv 0,3 \pmod{4}$,} \\[6pt]
\left\lceil \dfrac{m^2+m-10}{16} \right\rceil
& \text{if $m \equiv 1,2 \pmod{4}$,}
\end{cases}
\]
and that
\[
\left\lfloor \frac{m(m+1)}{4} \right\rfloor=
\begin{cases}
\dfrac{m^2+m}{4}
& \text{if $m \equiv 0,3 \pmod{4}$,} \\[6pt]
\dfrac{m^2+m-2}{4}
& \text{if $m \equiv 1,2 \pmod{4}$.}
\end{cases}
\]
Hence,
\[
\left\lfloor \frac{m(m+1)}{4} \right\rfloor \le 4N(m)+2.
\]
This proves (\ref{tildegafukumareru}).

Next, we prove that
\begin{equation}\label{1/2gafuku}
  \frac{1}{2}\widetilde{\mathcal{Z}}_{\ZZ}^{m-1}
  \subset c(\mathcal{K}_{\ZZ}^{m+4N(m)}).
\end{equation}

Since $\widetilde{\mathcal{Z}}_{\ZZ}^{1}=\{0\}$, the claim is obvious when $m=2$.
Assume that $m>2$. Since (\ref{tildegafukumareru}) holds with $m$ replaced by $m-1$, Lemma~\ref{1/2} implies that
\[
\frac{1}{2}\widetilde{\mathcal{Z}}_{\ZZ}^{m-1}
\subset c(\mathcal{K}_{\ZZ}^{m+4N(m-1)}).
\]

Furthermore, for $m>2$,
\[
(m-1)^2 + (m-1) - 8
=
m^2-m-8
<
m^2+m-10.
\]
Therefore, if $m\equiv1\pmod4$, then $N(m)\ge N(m-1)$.
For $m\equiv0,2,3\pmod4$, it follows immediately from the definition that
$N(m)\ge N(m-1)$.
Hence, $N(m)\ge N(m-1)$ for all $m>2$.
Therefore, by Lemma~\ref{1/2},
\[
c(\mathcal{K}_{\ZZ}^{m+4N(m-1)})
\subset
c(\mathcal{K}_{\ZZ}^{m+4N(m)}),
\]
and (\ref{1/2gafuku}) follows.
\end{proof}
\section*{Acknowledgments}
The author would like to express sincere gratitude to Professor Hidekazu Furusho for his valuable comments and helpful suggestions in the preparation of this paper. The author is also deeply grateful to Mr. Kelvin Ritland for carefully reading an earlier version of this manuscript and providing valuable advice.

\end{document}